\RequirePackage{fix-cm}
\documentclass[smallextended]{svjour3}       
\smartqed  
\usepackage{graphicx}
\usepackage{amsmath}
\usepackage{amssymb}
\usepackage{enumerate}

\begin{document}

\title{Moving frames and the characterization of curves that lie on a surface\footnote{The final publication is available at Springer via\\ \url{http://dx.doi.org/10.1007/s00022-017-0398-7}; (J. Geom. 2017)}}



\author{Luiz C. B. da Silva}


\institute{L. C. B. da Silva - orcid.org/0000-0002-2702-9976 \at
              Departamento de Matem\'atica, Universidade Federal de Pernambuco, \\
              50670-901, Recife, Pernambuco, Brazil \\
              \email{luizsilva@dmat.ufpe.br}           
}

\date{Received: date / Accepted: date}

\maketitle

\begin{abstract}
In this work we are interested in the characterization of curves that belong to a given surface. To the best of our knowledge, there is no known general solution to this problem. Indeed, a solution is only available for a few examples: planes, spheres, or cylinders. Generally, the characterization of such curves, both in Euclidean ($E^3$) and in Lorentz-Minkowski ($E_1^3$) spaces, involves an ODE relating curvature and torsion. However, by equipping a curve with a relatively parallel moving frame, Bishop was able to characterize spherical curves in $E^3$ through a linear equation relating the coefficients which dictate the frame motion. Here we apply these ideas to surfaces that are implicitly defined by a smooth function, $\Sigma=F^{-1}(c)$, by reinterpreting the problem in the context of the metric given by the Hessian of $F$, which is not always positive definite. So, we are naturally led to the study of curves in $E_1^3$. We develop a systematic approach to the construction of Bishop frames by exploiting the structure of the normal planes induced by the casual character of the curve, present a complete characterization of spherical curves in $E_1^3$, and apply it to characterize curves that belong to a non-degenerate Euclidean quadric. We also interpret the casual character that a curve may assume when we pass from $E^3$ to $E_1^3$ and finally establish a criterion for a curve to lie on a level surface of a smooth function, which reduces to a linear equation when the Hessian is constant. 
\keywords{Moving frame \and Curve on surface \and Spherical curve \and Level surface \and Euclidean space \and Lorentz-Minkowski space}
\subclass{53A04 \and 53A05 \and 53B30}
\end{abstract}

\section{Introduction}
\label{intro}

The study of curves is an important chapter of geometry. Besides its intrinsic importance, curves also play a major role in the analysis of surfaces, and manifolds in general. In particular, the consideration of curves that belong to a given surface, such as planar or spherical curves \cite{ArnoldRMS1995,ChmutovOregon2006,LinJDG1996}, may be of great value. In this respect an interesting problem is {\it``how can we characterize those (spatial) curves that belong to a certain surface $\Sigma$?"}. Despite the simplicity to formulate the problem, there is no known general solution to it except for a few cases: namely, when $\Sigma$ is a plane \cite{Struik}, a sphere \cite{WongMonatshMath,Struik}, or a cylinder \cite{StarostinMonatshMath}. The solution for planar curves is quite trivial once we introduce the Frenet frame: the torsion must vanish. On the other hand, the characterization for spherical curves generally involves an ODE relating the curvature function and the torsion \cite{WongMonatshMath}, while the characterization for cylindrical curves involves a system of algebro-differential equations \cite{StarostinMonatshMath}. In the 70's, through the idea of equipping a curve with a relatively parallel moving frame, Bishop was able to characterize spherical curves through a linear equation relating the coefficients that dictate the frame motion \cite{BishopMonthly}. The coefficients of such a Bishop frame admit a simple geometric interpretation and, besides its impact on the study of spherical curves, a Bishop frame also has the advantage of being globally defined even if a curve has points of zero curvature \cite{BishopMonthly}. Naturally, it also finds applications in problems which make use of frames along curves, such as in rotation-minimizing frames in rigid body dynamics \cite{FaroukiCAGD2014}, computer graphics and visualization \cite{HansonTechrep1995}, robotics \cite{WebsterIJRR2010}, quantum waveguides \cite{HaagAHP2015}, integrable systems \cite{SandersMMJ2003}, and also in mathematical biology in the study of DNA \cite{ChirikjianBST2013,ClauvelinJCTC2012} and protein folding \cite{HuPRE2011}, just to name a few. 

In the quest of spherical curves we should not restrict ourselves to the context of an Euclidean ambient space, $(E^3,\langle\cdot,\cdot\rangle\,)$. Indeed, we can consider the more general setting of a Lorentz-Minkowski space, $(E_1^3,(\cdot,\cdot)\,)$, where one has to deal with three types of spheres: pseudo-spheres $\mathbb{S}_1^2(P;r)=F_P^{-1}(r^2)$; pseudo-hyperbolic spaces $\mathbb{H}_0^2(P;r)=F_P^{-1}(-r^2)$; and light-cones $\mathcal{C}^2(P)=F_P^{-1}(0)$, where $F_P(x)=(x-P,x-P)$ and $(\cdot,\cdot)$ has index 1. Indeed, it is possible to find characterizations of some classes of spherical curves scattered among a few papers: pseudo-spherical \cite{BektasBMMS1998,IlarslamJII-PP2003,PekmenMM1999,Petrovic-TorgasevMM2000,Petrovic-TorgasevMV2001} and pseudo-hyperbolic curves \cite{IlarslamJII-PP2003,Petrovic-TorgasevKJM2000} via Frenet frame; and also curves on light-cones \cite{ErdoganJST2009,LiuRM2011,LiuJG2016} by exploiting conformal invariants and the concept of cone curvature \cite{LiuBAG2004}. It is also possible to find constructions of Bishop frames on curves in $E_1^3$ for spacelike curves \cite{BukcuCFSUA2008,BukcuSJAM2010,LowJGSP2012,OzdemirMJMS2008} with a non-lightlike normal, and timelike curves \cite{KaracanSDUJS2008,LowJGSP2012,OzdemirMJMS2008}, along with several characterizations of spherical curves through a linear equation via Bishop frames \cite{BukcuCFSUA2008,BukcuSJAM2010,KaracanSDUJS2008,OzdemirMJMS2008}. All the above mentioned studies in $E_1^3$ have in common that much attention is paid on the possible combinations of casual characters of the tangent and normal vectors, which makes necessary the consideration of several instances of the investigation of Bishop frames and spherical curves. Moreover, none of them take into account the possibility of a lightlike tangent or a lightlike normal. Naturally, this reflects in the incompleteness of the available characterizations of spherical curves in $E_1^3$.

Here we apply these ideas in order to characterize those spatial curves that belong to surfaces implicitly defined by a smooth function, $\Sigma=F^{-1}(c)$, by reinterpreting the problem in the new geometric setting of an inner product induced by the Hessian, $\mbox{Hess}\,F=\partial^2F/\partial x^i\partial x^j$. Although simple, this idea will prove to be very useful. Moreover, since a Hessian may fail to be positive definite, one is naturally led to the study of the differential geometry of curves in Lorentz-Minkowski spaces. In this work, we then present a systematic approach to moving frames on curves in $E_1^3$. The turning point is that one should exploit the casual character of the tangent vector and the induced casual character on the normal plane only. In this way, we are able to furnish a systematic approach to the construction of Bishop frames in $E_1^3$. This formalism allows us to give a complete characterization of spherical curves in $E_1^3$. Finally, we present a necessary and sufficient criterion for a curve to lie on a level surface of a smooth function. More precisely, we present a functional relationship involving the coefficients of a Bishop frame with respect to the Hessian metric along a curve on $\Sigma=F^{-1}(c)$, which reduces to a linear relation when $\mbox{Hess}\,F$ is constant. In this last case, we are able to characterize spatial curves that belong to a given non-degenerate Euclidean quadric $\mathcal{Q}=\{x:\langle B(x-P),(x-P)\rangle=\rho\}$, $\rho\in\mathbb{R}$ constant, by using $(\cdot,\cdot)=\langle B\cdot,\cdot\rangle$. We also furnish an interpretation for the casual character that a curve may assume when we pass from $E^3$ to $E_1^3$, which also allows us to understand why certain types of curves do not exist on a given quadric or on a given Lorentzian sphere, if we reinterpret the problem from $E_1^3$ in $E^3$. To the best of our knowledge, this is the first time that this characterization problem is considered in a general context.

This work is organized as follows. In Section 2 we review the construction of relatively parallel moving frames in Euclidean space according to Bishop. Section 3 is devoted to moving frames in Lorentz-Minkowski spaces: subsection 3.1 to Frenet frames in $E_1^3$; subsections 3.2 and 3.3 to Bishop frames along space- and timelike curves and their geometric interpretations, respectively; and, subsection 3.4, to null frames along lightlike curves. In Section 4 we characterize spherical curves in Lorentz-Minkowski spaces, i.e., curves on pseudo-spheres, pseudo-hyperbolic spaces, and light-cones. In Section 5 we characterize curves on a non-degenerate Euclidean quadric. In Section 6 we present a characterization of curves that lie on a regular level surface. And finally, in Section 7 we present our conclusions along with some open problems and directions of future research. 

\section{Moving frames on curves in $E^3$}
\label{sec:MovingFrameCurves}
 
Let us denote by $E^3$ the $3d$ Euclidean space, i.e., $\mathbb{R}^3$ equipped with the standard metric $\langle\cdot,\cdot\rangle$. Given a regular curve $\alpha:I\rightarrow E^3$ parametrized by arc-length, the usual way to introduce a moving frame along it is by means of the Frenet frame $\{\mathbf{t},\mathbf{n},\mathbf{b}\}$ \cite{Struik}. However, we can also consider any other adapted orthonormal moving frame $\{\mathbf{e}_0(s),\mathbf{e}_1(s),\mathbf{e}_2(s)\}$ along $\alpha(s)$, i.e., $\mathbf{e}_0\propto \mathbf{t}$ and  $\langle\mathbf{e}_i,\mathbf{e}_j\rangle=\delta_{ij}$. The equation of motion of such a moving frame is given by a skew-symmetric $3\times3$ matrix. For the Frenet frame one of the entries of this matrix is zero and the other two are the curvature function $\kappa$ and the torsion $\tau$:
\begin{equation}
\frac{{\rm d}}{{\rm d}s}\left(
\begin{array}{c}
\mathbf{t}\\
\mathbf{n}\\
\mathbf{b}\\
\end{array}
\right)=\left(
\begin{array}{ccc}
0 & \kappa & 0\\
-\kappa & 0 & \tau\\
0 & -\tau & 0\\ 
\end{array}
\right)\left(
\begin{array}{c}
\mathbf{t}\\
\mathbf{n}\\
\mathbf{b}\\
\end{array}
\right).\label{eq::FrenetEqs}
\end{equation}

By introducing the notion of a relatively parallel vector field, Bishop considered a moving frame $\{\mathbf{t},\mathbf{n}_1,\mathbf{n}_2\}$, where $\mathbf{n}_i$ are normal vectors to the unit tangent $\mathbf{t}$, whose equation of motion is \cite{BishopMonthly} 
\begin{equation}
\frac{{\rm d}}{{\rm d}s}\left(
\begin{array}{c}
\mathbf{t}\\
\mathbf{n}_1\\
\mathbf{n}_2\\
\end{array}
\right)=\left(
\begin{array}{ccc}
0 & \kappa_{1} & \kappa_{2}\\
-\kappa_{1} & 0 & 0\\
-\kappa_{2} & 0 & 0\\ 
\end{array}
\right)\left(
\begin{array}{c}
\mathbf{t}\\
\mathbf{n}_1\\
\mathbf{n}_2\\
\end{array}
\right).\label{eq::BishopEqs}
\end{equation}

The coefficients $\kappa_1$ and $\kappa_2$ relate with the curvature function and torsion according to \cite{BishopMonthly}
\begin{equation}
\left\{
\begin{array}{c}
\kappa_1 = \kappa\cos\theta\\
\kappa_2 = \kappa\sin\theta\\
\theta'= \tau\\
\end{array}
\right..
\end{equation}
\begin{remark}
A vector field $\mathbf{e}(s)$ along $\alpha(s)$ is {\it relatively parallel} if the derivative of its normal component, $\mathbf{e}^{\perp}$, is a multiple of the unit tangent vector, i.e., ${\rm d}\mathbf{e}^{\perp}/{\rm d}s=\eta(s)\mathbf{t}(s)$, and the tangent component is a constant multiple of $\mathbf{t}$ \cite{BishopMonthly}.
\end{remark}
\begin{remark}
Such a frame may be also named {\it rotation minimizing frame}, since $\mathbf{n}_i$ does not rotate around $\mathbf{t}$. In addition, it can be proved that $\mathbf{n}_i$ is parallel transported along $\alpha(s)$ with respect to the normal connection of the curve \cite{Etayo2016}. Observe that for a closed curve, $\alpha(s_i)=\alpha(s_f)$, $\mathbf{n}_1(s_i)$ will differ from $\mathbf{n}_1(s_f)$, by an angular amount of $\Delta \theta = \int_{s_i}^{s_f} \tau(x){\rm d}x$.
\end{remark}

An advantage of such a relatively parallel moving frame, or {\it Bishop frame} for short\footnote{This frame has been independently discovered several times \cite{DaCostaPRA1981,TangIEEE1970}, e.g., in the physics literature it is sometimes named as the Tang frame. However, Bishop seems to be the first to exploit the geometric implications of such frames.}, is that it can be globally defined even if the curve is degenerate, i.e., if the curvature $\kappa$ vanishes at some points \cite{BishopMonthly}. Furthermore, it also allows for a simple characterization of spherical curves:
\begin{theorem}[Bishop \cite{BishopMonthly}]
A $C^2$ regular curve lies on a sphere if and only if its normal development, i.e., the curve $(\kappa_1(s),\kappa_2(s))$, lies on a line not passing through the origin. Moreover, the distance of this line from the origin, $d$, and the radius of the sphere, $r$, are reciprocals: $r=d^{-1}$.
\label{theo:BishopCharacSpherericalCurves} 
\end{theorem}
\begin{remark}
Straight lines passing through the origin characterize planar curves which are not spherical \cite{BishopMonthly}.
\end{remark}
Finally, Bishop frames are not uniquely defined. Indeed, any rotation of $\mathbf{n}_1$ and $\mathbf{n}_2$ still gives two relatively parallel fields, i.e., there is an ambiguity associated with the group $SO(2)$ acting on the normal plane. However, the coefficients still determine a curve up to rigid motions \cite{BishopMonthly}. Moreover, $\kappa$-constant curves are represented in the normal development plane by circles centered at the origin with radius $\kappa$ \cite{BishopMonthly}, which can be seen as the orbits of the symmetry group $SO(2)$.

In the following we shall extend this formalism in order to present a way of building Bishop frames along curves in $E_1^3$ and then apply it to furnish a unified approach to the characterization of spherical curves in $E_1^3$ (i.e., curves on pseudo-spheres, pseudo-hyperbolic spaces, and light-cones), curves on quadrics in $E^3$, and finally characterize curves that lie on level surfaces of a smooth function by reinterpreting the problem in a new geometric setting.

\section{Moving frames on curves in $E_{1}^3$}
\label{sec:MovingFrameCurvesE3_1}

Let us denote by $E_{1}^3$ the vector space $\mathbb{R}^3$ equipped with a pseudo-metric $(\cdot,\cdot)$ of index $1$. In fact, the concepts below, and the construction of Bishop-like frames as well, are still valid in the context of a 3-dimensional semi-Riemannian manifold \cite{ONeill}, but to help intuition, the reader may keep in mind the particular setting of $\mathbb{R}^3$ equipped with the standard Minkowski metric, i.e. $(x,y)=x_1y_1+x_2y_2-x_3y_3$. Naturally, in a more general context, the derivative of a vector field along a curve should be understood as a covariant derivative. 

Before discussing the moving frame method on curves in $E_1^3$, let us introduce some terminology and geometric properties associated with $E_1^3$ (for more details, we refer to \cite{LopesIEJG2014,ONeill}).

One property that makes the geometry in Lorentz-Minkowski spaces $E_1^3$ more difficult and richer than the geometry in $E^3$ is that curves and vector subspaces may assume different casual characters:
\begin{definition}
A vector $v\in E_1^3$ assumes one of the following {\it casual characters}:
\begin{enumerate}[(a)]
\item $v$ is {\it spacelike}, if $(v,v)>0$ or $v=0$;
\item $v$ is {\it timelike}, if $(v,v)<0$;
\item $v$ is {\it lightlike}, if $(v,v)=0$ and $v\not=0$.
\end{enumerate}
\end{definition}

The inner product $(\cdot,\cdot)$ induces a pseudo-norm defined by $\Vert x\Vert=\sqrt{\vert(x,x)\vert}$. Given a vector subspace $U\subseteq\mathbb{R}^3$, we define the orthogonal complement $U^{\perp}$ in the usual way: $U^{\perp}=\{v\in E_1^3:\forall\,u\in U,\,(v,u)=0\}$. Moreover, we can consider the restriction of $(\cdot,\cdot)$ to $U$, $(\cdot,\cdot)|_{U}$.
\begin{definition}
Let $U$ be a vector subspace, then
\begin{enumerate}[(a)]
\item $U$ is {\it spacelike} if $(\cdot,\cdot)|_{U}$ is positive definite;
\item $U$ is {\it timelike} if $(\cdot,\cdot)|_{U}$ has index 1;
\item $U$ is {\it lightlike} if $(\cdot,\cdot)|_{U}$ is degenerate.
\end{enumerate}
\end{definition}
We have the following useful properties related to the casual characters of vector subspaces:
\begin{proposition}
Let $U\subseteq E_1^3$ be a vector subspace. Then,
\begin{enumerate}[(i)]
\item $\dim U^{\perp} = 3-\dim U$ and $(U^{\perp})^{\perp}=U$;
\item $U$ is lightlike if and only if $U^{\perp}$ is lightlike;
\item $U$ is spacelike (timelike) if and only if $U^{\perp}$ is timelike (spacelike).
\item $U$ is lightlike if and only if $U$ contains a lightlike vector but not a timelike one. Moreover, $U$ admits an orthogonal basis formed  by a lightlike and a spacelike vectors.
\end{enumerate}
\end{proposition}

Given two vectors $u,v\in E_1^3$, the Lorentzian vector product, denoted by $u\times v$, is the only vector that satisfies
\begin{equation}
\forall\,w\in E_1^3,\,(u\times v,w)=\det(u,v,w),\label{eq::LorentzVectorProd}
\end{equation}
where the columns of $(u,v,w)$ are formed by the entries of $u,v$, and $w$.

From these definitions, we say that a curve $\alpha:I\to E_1^3$ is spacelike, timelike, or lightlike, if its velocity vector $\alpha'$ is spacelike, timelike, or lightlike, respectively. Analogously, we say that a surface is spacelike, timelike, or lightlike, if its tangent planes are spacelike, timelike, or lightlike, respectively.

If a curve is lightlike we can not define an arc-length parameter (in $E^3$ this is always possible). In this case, one must introduce the notion of a {\it pseudo arc-length parameter}, i.e., a parameter $s$ such that $(\alpha''(s),\alpha''(s))=1$. More precisely, if $\alpha$ is a lightlike curve and $(\alpha'',\alpha'')\not=0$ (otherwise $\alpha''$ and $\alpha'$ will be linearly dependent and the curve is a straight line), we define the {\it pseudo arc-length parameter} as
\begin{equation}
s = \int_a^t\Vert\alpha''(u)\Vert\,{\rm d}u\,.\label{eq:PseudoArclength}
\end{equation}
On the other hand, if $\alpha$ is not a lightlike curve, then the {\it arc-length parameter} is defined as usual
\begin{equation}
s = \int_a^t\Vert\alpha'(u)\Vert\,{\rm d}u\,.\label{eq:arclength}
\end{equation}
In the following we will assume every curve to be parametrized by the arc-length or pseudo arc-length parameter.

\subsection{Frenet frame in $E_1^3$}

The study of the local properties of a curve $\alpha\subset E_1^3$ in a Frenet frame fashion becomes quite cumbersome due to the various possibility for the casual characters of the tangent and its derivative. In essence, there is a construction for each combination of the casual characters of $\mathbf{t}$ and $\mathbf{t}'$.

Let $\mathbf{t}(s)=\alpha'(s)$ be the (unit) tangent and $s$ the arc- or pseudo arc-length parameter. If $\mathbf{t}'$ is not a lightlike vector, let $\mathbf{n}=\mathbf{t}'/\Vert\mathbf{t}'\Vert$ be the normal vector. We shall denote by $\epsilon=(\mathbf{t},\mathbf{t})$ and $\eta=(\mathbf{n},\mathbf{n})$ the parameters that enclose the casual character of the tangent and normal vectors. If $\mathbf{t}$ and $\mathbf{n}$ are not lightlike, then
\begin{equation}
\frac{{\rm d}}{{\rm d}s}\left(
\begin{array}{c}
\mathbf{t}\\
\mathbf{n}\\
\mathbf{b}\\
\end{array}
\right)=\left(
\begin{array}{ccc}
0 & \eta\,\kappa & 0\\
-\epsilon\,\kappa & 0 & -\epsilon\eta\,\tau\\
0 & -\eta\,\tau & 0\\ 
\end{array}
\right)\left(
\begin{array}{c}
\mathbf{t}\\
\mathbf{n}\\
\mathbf{b}\\
\end{array}
\right)=\left(
\begin{array}{ccc}
0 & \kappa & 0\\
-\kappa & 0 & \tau\\
0 & -\tau & 0\\ 
\end{array}
\right)E_{\mathbf{t},\mathbf{n},\mathbf{b}}\left(
\begin{array}{c}
\mathbf{t}\\
\mathbf{n}\\
\mathbf{b}\\
\end{array}
\right),\label{eq::FrenetEqsInE13}
\end{equation}
where $\mathbf{b}=\mathbf{t}\times\mathbf{n}$, and $\kappa = (\mathbf{t}',\mathbf{n})$ and $\tau=(\mathbf{n}',\mathbf{b})$ are the curvature function and torsion of $\alpha$, respectively\footnote{Our definition for $\kappa$ is slightly different from that of L\'opez \cite{LopesIEJG2014}. Indeed, despite the fact that the definition is formally identical to the Euclidean version, our $\kappa$ is a signaled curvature and its sign encloses the casual character of the curve in a natural manner.}. Here $E_{\mathbf{t},\mathbf{n},\mathbf{b}}=\mbox{diag}(\epsilon,\eta,-\epsilon\eta)=[(\mathbf{e}_i,\mathbf{e}_j)]_{ij}$ denotes the matrix associated with the frame $\{\mathbf{e}_0=\mathbf{t},\mathbf{e}_1=\mathbf{n},\mathbf{e}_2=\mathbf{b}\}$.

If $\mathbf{t}$ is spacelike and $\mathbf{t}'$ is lightlike, we define $\mathbf{n}=\mathbf{t}'$, while $\mathbf{b}$ is the unique lightlike vector orthonormal to $\mathbf{t}$ that satisfies $(\mathbf{n},\mathbf{b})=-1$. The Frenet equations are 
\begin{equation}
\frac{{\rm d}}{{\rm d}s}\left(
\begin{array}{c}
\mathbf{t}\\
\mathbf{n}\\
\mathbf{b}\\
\end{array}
\right)=\left(
\begin{array}{ccc}
0 & \,1 & 0\\
0 & \,\tau & 0\\
1 & \,0 & -\tau\\ 
\end{array}
\right)\left(
\begin{array}{c}
\mathbf{t}\\
\mathbf{n}\\
\mathbf{b}\\
\end{array}
\right)=E_{\mathbf{t},\mathbf{n},\mathbf{b}}\left(
\begin{array}{ccc}
0 & 1 & 0\\
-1 & 0 & \tau\\
0 & -\tau & 0\\ 
\end{array}
\right)\left(
\begin{array}{c}
\mathbf{t}\\
\mathbf{n}\\
\mathbf{b}\\
\end{array}
\right),\label{eq::FrenetEqsInE13tprimelightlike}
\end{equation}
where $\tau=-(\mathbf{n}',\mathbf{b})$ is the pseudo-torsion. Here $E_{\mathbf{t},\mathbf{n},\mathbf{b}}=[(\mathbf{e}_i,\mathbf{e}_j)]_{ij}$ denotes the matrix associated with the null frame $\{\mathbf{e}_0=\mathbf{t},\mathbf{e}_1=\mathbf{n},\mathbf{e}_2=\mathbf{b}\}$.

Finally, if $\mathbf{t}$ is lightlike, we define $\mathbf{n}=\mathbf{t}'$ (we assume this normal vector to be spacelike, otherwise $\alpha$ is a straight line), while $\mathbf{b}$ is the unique lightlike vector that satisfies $(\mathbf{n},\mathbf{b})=0$ and $(\mathbf{t},\mathbf{b})=-1$. The Frenet equations are then
\begin{equation}
\frac{{\rm d}}{{\rm d}s}\left(
\begin{array}{c}
\mathbf{t}\\
\mathbf{n}\\
\mathbf{b}\\
\end{array}
\right)=\left(
\begin{array}{ccc}
0 & 1 & 0\\
-\tau & 0 & 1\\
0 & -\tau & 0\\ 
\end{array}
\right)\left(
\begin{array}{c}
\mathbf{t}\\
\mathbf{n}\\
\mathbf{b}\\
\end{array}
\right)=\left(
\begin{array}{ccc}
0 & 1 & 0\\
-1 & 0 & \tau\\
0 & -\tau & 0\\ 
\end{array}
\right)E_{\mathbf{t},\mathbf{n},\mathbf{b}}\left(
\begin{array}{c}
\mathbf{t}\\
\mathbf{n}\\
\mathbf{b}\\
\end{array}
\right),\label{eq::FrenetEqsInE13tangentlightlike}
\end{equation}
where $\tau=(\mathbf{n}',\mathbf{b})$ is the pseudo-torsion. Here $E_{\mathbf{t},\mathbf{n},\mathbf{b}}=[(\mathbf{e}_i,\mathbf{e}_j)]_{ij}$ denotes the matrix associated with the null frame $\{\mathbf{e}_0=\mathbf{t},\mathbf{e}_1=\mathbf{n},\mathbf{e}_2=\mathbf{b}\}$.
\begin{remark}
In $E^3$ the coefficient matrix of a Frenet frame is always skew-symmetric. On the other hand, this does not happen in $E_1^3$ \cite{LowJGSP2012}. However, the above expressions show that the coefficient matrix can be obtained from a skew-symmetric matrix through a right-multiplication, or a left one if $\mathbf{t}'$ is lightlike, by the matrix $E_{\mathbf{t},\mathbf{n},\mathbf{b}}=[(\mathbf{e}_i,\mathbf{e}_j)]_{ij}$ associated with the respective Frenet frame $\{\mathbf{e}_0=\mathbf{t},\mathbf{e}_1=\mathbf{n},\mathbf{e}_2=\mathbf{b}\}$ in $E_1^3$. This skew-symmetric matrix is precisely the coefficient matrix that we would obtain for a Frenet frame in $E^3$. Let us mention that when $\mathbf{t}'$ is lightlike it does not mean that the curvature function is $\kappa=1$; a curvature is not well defined for such curves \cite{LopesIEJG2014}.
\end{remark}
\begin{remark}
In the following, when discussing Bishop frames in $E_1^3$ along non-lightlike curves and null frames along lightlike curves, we will see that the coefficient matrix can be obtained from a skew-symmetric matrix (precisely the matrix that we would obtain for a Bishop frame in $E^3$) through a right-multiplication by the matrix associated with a convenient basis.
\end{remark}

\subsection{Relatively parallel moving frames along spacelike or lightlike curves}

A quite complete and systematic approach to the problem of the existence of Bishop-like frames along curves in $E_1^3$ was presented by \"Ozdemir and Ergin \cite{OzdemirMJMS2008}, where they build Bishop-like frames on timelike and spacelike curves with a non-lightlike normal. However, as in the Frenet frame case, they also paid much attention to the casual character of $\mathbf{t}'$. Here, we show that one must exploit the structure of the normal plane inherited from the casual character of $\mathbf{t}$ in order to build a unified treatment of the problem. More precisely, instead of considering the problem for each combination of the casual character of $\mathbf{t}$ and $\mathbf{t}'$, one must pay attention to the symmetry associated with the problem, which is reflected in an ambiguity in the definition of a Bishop frame. The study of moving frames along curves in $E_1^3$ is then divided in three cases only: (i) timelike curves; (ii) spacelike curves; and (iii) lightlike curves. As a direct consequence, the characterization of spherical curves can be split along three Theorems only.

\begin{definition}
A vector field $\mathbf{e}(s)$ along a curve $\alpha:I\to E_1^3$ is {\it relatively parallel} if the derivative of its normal component is a multiple of the unit tangent vector $\mathbf{t}=\alpha'$ and its tangent component is a constant multiple of $\mathbf{t}$.
\end{definition}

Let $\alpha:I\to E_1^3$ be a timelike curve. Since $\mathbf{t}$ is a timelike vector, the normal plane $N_{\alpha(s)}=\mbox{span}\{\mathbf{t}(s)\}^{\perp}$ is spacelike. To prove the existence of relatively parallel moving frames, let $\mathbf{x}_1$ and $\mathbf{x}_2=\mathbf{t}\times\mathbf{x}_1$ be an orthonormal basis of $N_{\alpha}$. The frame $\{\mathbf{t},\mathbf{x}_1,\mathbf{x}_2\}$ satisfies the following equations
\begin{equation}
\frac{{\rm d}}{{\rm d}s}\left(
\begin{array}{c}
\mathbf{t}\\
\mathbf{x}_1\\
\mathbf{x}_2\\
\end{array}
\right)=\left(
\begin{array}{ccc}
0 & p_{01} & p_{02}\\
p_{01} & 0 & p_{12}\\
p_{02} & -p_{12} & 0\\ 
\end{array}
\right)\left(
\begin{array}{c}
\mathbf{t}\\
\mathbf{x}_1\\
\mathbf{x}_2\\
\end{array}
\right)=\left(
\begin{array}{ccc}
0 & p_{01} & p_{02}\\
-p_{01} & 0 & p_{12}\\
-p_{02} & -p_{12} & 0\\ 
\end{array}
\right)E_{\mathbf{t},\mathbf{x}_1,\mathbf{x}_2}\left(
\begin{array}{c}
\mathbf{t}\\
\mathbf{x}_1\\
\mathbf{x}_2\\
\end{array}
\right),
\end{equation}
for some functions $p_{ij}$, where $E_{\mathbf{t},\mathbf{x}_1,\mathbf{x}_2}=[(\mathbf{e}_i,\mathbf{e}_j)]_{ij}$ denotes the matrix associated with the time-oriented frame $\{\mathbf{e}_0=\mathbf{t},\mathbf{e}_k=\mathbf{x}_k\}$. Let $\theta$ be a smooth function such that $\mathbf{x}=L\cos\theta\,\mathbf{x}_1+L\sin\theta\,\mathbf{x}_2$, where $L$ is a constant. Then,
\begin{equation}
\mathbf{x}'=L(p_{01}\cos\theta+p_{02}\sin\theta)\mathbf{t}+L(\theta'+p_{12})(-\sin\theta\mathbf{x}_1+\cos\theta\mathbf{x}_2).
\end{equation}
Thus, it follows that $\mathbf{x}$ is relatively parallel if and only if $\theta'+p_{12}=0$. By the existence of a solution $\theta(s)$ for any initial condition, this shows that relatively parallel vector fields do exist along timelike curves. Observe that Bishop frames are not unique. Indeed, any rotation of the normal vectors still gives two relatively parallel vector fields, i.e., there is an ambiguity associated with the group $SO(2)$.

On the other hand, if $\alpha:I\to E_1^3$ is a spacelike curve, $\mathbf{t}$ is a spacelike vector and then the normal plane $N_{\alpha(s)}=\mbox{span}\{\mathbf{t}(s)\}^{\perp}$ is timelike. In a Frenet frame fashion, the study is divided into three cases, depending on the casual character of $\mathbf{t}'\in N_{\alpha}$, i.e., if $\mathbf{t}'$ is a space-, time-, or lightlike vector. But, if we only take into account the structure of $N_{\alpha}$, this is no longer necessary.

To prove the existence of relatively parallel moving frames along spacelike curves, let $\mathbf{y}_1\in N_{\alpha}$ be a timelike vector and let $\mathbf{y}_2=\mathbf{t}\times\mathbf{y}_1$ be spacelike. Then, the frame $\{\mathbf{t},\mathbf{y}_1,\mathbf{y}_2\}$ is an orthonormal time-oriented basis of $E_1^3$ along $\alpha$. The frame $\{\mathbf{t},\mathbf{y}_1,\mathbf{y}_2\}$ satisfies the following equation of motion
\begin{equation}
\frac{{\rm d}}{{\rm d}s}\left(
\begin{array}{c}
\mathbf{t}\\
\mathbf{y}_1\\
\mathbf{y}_2\\
\end{array}
\right)=\left(
\begin{array}{ccc}
0 & -p_{01} & p_{02}\\
-p_{01} & 0 & p_{12}\\
-p_{02} & p_{12} & 0\\ 
\end{array}
\right)\left(
\begin{array}{c}
\mathbf{t}\\
\mathbf{y}_1\\
\mathbf{y}_2\\
\end{array}
\right)=\left(
\begin{array}{ccc}
0 & p_{01} & p_{02}\\
-p_{01} & 0 & p_{12}\\
-p_{02} & -p_{12} & 0\\ 
\end{array}
\right)E_{\mathbf{t},\mathbf{y}_1,\mathbf{y}_2}\left(
\begin{array}{c}
\mathbf{t}\\
\mathbf{y}_1\\
\mathbf{y}_2\\
\end{array}
\right),
\end{equation}
for some functions $p_{ij}$, where $E_{\mathbf{t},\mathbf{y}_1,\mathbf{y}_2}=[(\mathbf{e}_i,\mathbf{e}_j)]_{ij}$ denotes the matrix associated with the time-oriented frame $\{\mathbf{e}_0=\mathbf{t},\mathbf{e}_k=\mathbf{y}_k\}$. Let $\theta$ be a smooth function such that $\mathbf{y}=L\cosh\theta\,\mathbf{y}_1+L\sinh\theta\,\mathbf{y}_2$, where it is used hyperbolic trigonometric functions because the normal plane is timelike. Then, we have
\begin{equation}
\mathbf{y}'=L(-p_{01}\cosh\theta-p_{02}\sinh\theta)\mathbf{t}+L(\theta'+p_{12})(\sinh\theta\mathbf{y}_1+\cosh\theta\mathbf{y}_2).
\end{equation}
Thus, it follows that $\mathbf{y}$ is relatively parallel if and only if $\theta'+p_{12}=0$. By the existence of a solution $\theta(s)$ for any initial condition, this shows that relatively parallel vector fields do exist along spacelike curves. As in the previous case, observe that Bishop frames are not unique. Indeed, any (hyperbolic) rotation of the normal vectors still gives two relatively parallel vector fields, i.e., there is an ambiguity associated with the group $SO_1(2)$, which is a
component of the symmetry group of a Lorentzian plane $E^2_1$ \cite{LopesIEJG2014,ONeill}. 

When $\mathbf{n}$ has a distinct casual character from that of $\mathbf{n}_1$, then we can not obtain $\mathbf{n},\mathbf{b}$ from a $SO_1(2)$-rotation of $\mathbf{n}_1,\mathbf{n}_2$, i.e., there exists no $ M\in SO_1(2)$ such that $M(\mathbf{n})=\mathbf{n}_1$ and $M(\mathbf{b})=\mathbf{n}_2$. In this case, we must first exchange $\mathbf{n}_1$ and $\mathbf{n}_2$ and then rotate them \cite{OzdemirMJMS2008}. However, we can still read the information about the casual character of $\mathbf{n}$, including the lightlike case, from the ``circles'' of the normal plane, i.e., the orbits of $O_1(2)$, see figure 1 and Proposition \ref{prop::geomNormalDevelopm} below.

Now we put together the above mentioned existence results of relatively parallel vectors on non-lightlike curves. Let $\{\mathbf{n}_1,\mathbf{n}_2\}$ be a basis for $N_{\alpha}$ formed by relatively parallel vectors such that
\begin{equation}
\mathbf{n}'_i(s) = -\epsilon \kappa_i\,\mathbf{t}(s),
\end{equation}
where $\epsilon = (\mathbf{t},\mathbf{t})=\pm1$ and we have defined the Bishop curvatures
\begin{equation}
\kappa_i = (\mathbf{t}',\mathbf{n}_i),\,i=1,2\,.
\end{equation}
Then, defining $\epsilon_1=(\mathbf{n}_1,\mathbf{n}_1)=\pm1$, we can write the following equation of motion
\begin{equation}
\frac{{\rm d}}{{\rm d}s}\left(
\begin{array}{c}
\mathbf{t}\\
\mathbf{n}_1\\
\mathbf{n}_2\\
\end{array}
\right)=\left(
\begin{array}{ccc}
0 & \epsilon_1\kappa_{1} & \kappa_{2}\\
-\epsilon\kappa_{1} & 0 & 0\\
-\epsilon\kappa_{2} & 0 & 0\\ 
\end{array}
\right)\left(
\begin{array}{c}
\mathbf{t}\\
\mathbf{n}_1\\
\mathbf{n}_2\\
\end{array}
\right)=\left(
\begin{array}{ccc}
0 & \kappa_{1} & \kappa_{2}\\
-\kappa_{1} & 0 & 0\\
-\kappa_{2} & 0 & 0\\ 
\end{array}
\right)E_{\mathbf{t},\mathbf{n}_1,\mathbf{n}_2}\left(
\begin{array}{c}
\mathbf{t}\\
\mathbf{n}_1\\
\mathbf{n}_2\\
\end{array}
\right),\label{eq::GenBishopEqsNonlightCurves}
\end{equation}
 where $E_{\mathbf{t},\mathbf{n}_1,\mathbf{n}_2}=[(\mathbf{e}_i,\mathbf{e}_j)]_{ij}$ denotes the matrix associated with the time-oriented frame $\{\mathbf{e}_0=\mathbf{t},\mathbf{e}_k=\mathbf{n}_k\}$. The numbers $\epsilon$ and $\epsilon_1$ determine the casual character of $\mathbf{t}$ and $\mathbf{n}_1$, respectively, and since $\mathbf{n}_2=\mathbf{t}\times\mathbf{n}_1$, we have $\epsilon_2=(\mathbf{n}_2,\mathbf{n}_2)=-\epsilon\epsilon_1=+1$. So, in this case $E_{\mathbf{t},\mathbf{n}_1,\mathbf{n}_2}=\mbox{diag}(\epsilon,\epsilon_1,-\epsilon\epsilon_1)$.

\subsection{Geometry of the normal development of spacelike and timelike curves}

\begin{figure*}[tbp]
\centering
    {\includegraphics[width=0.33\linewidth]{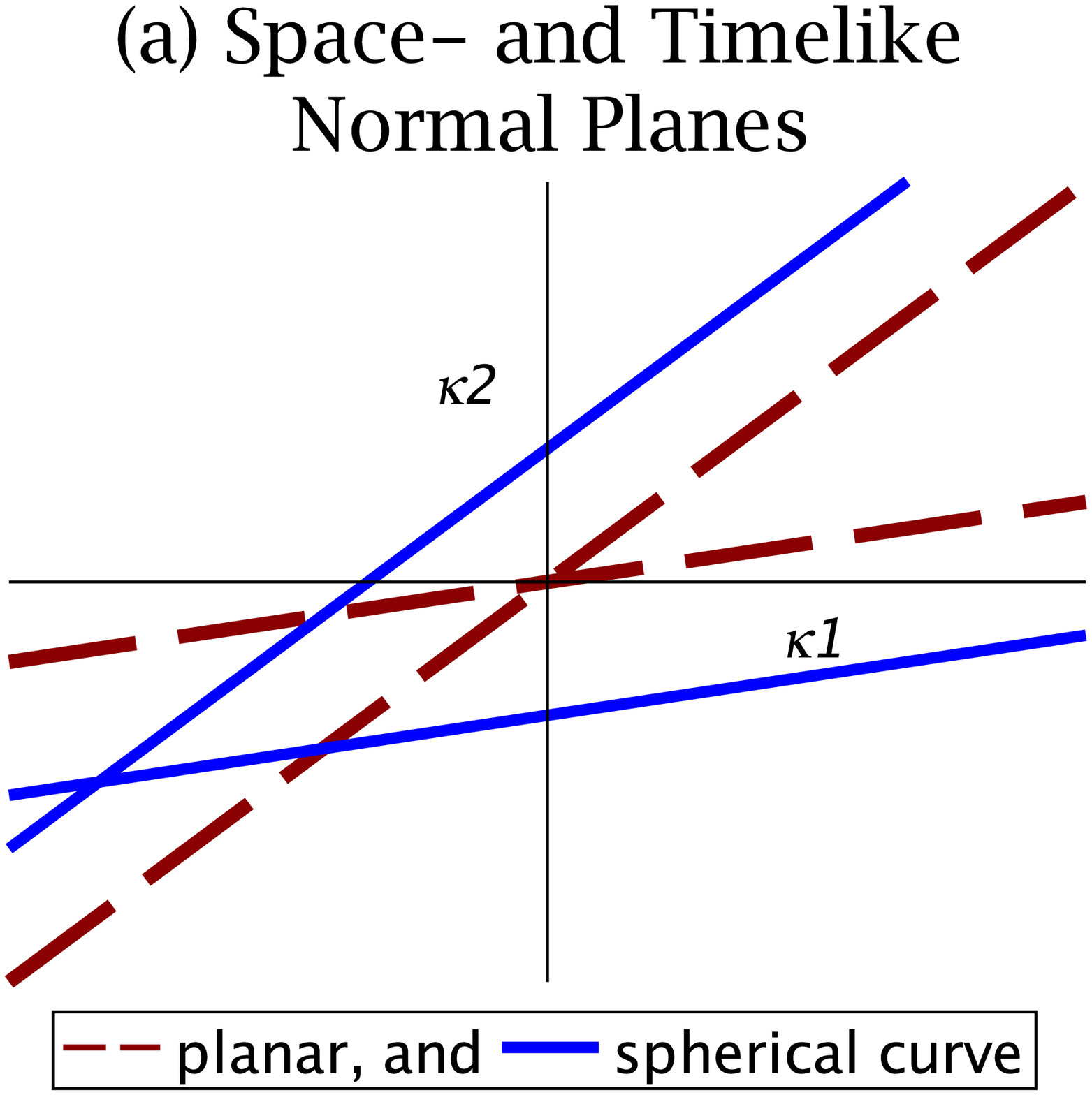}}
    {\includegraphics[width=0.32\linewidth]{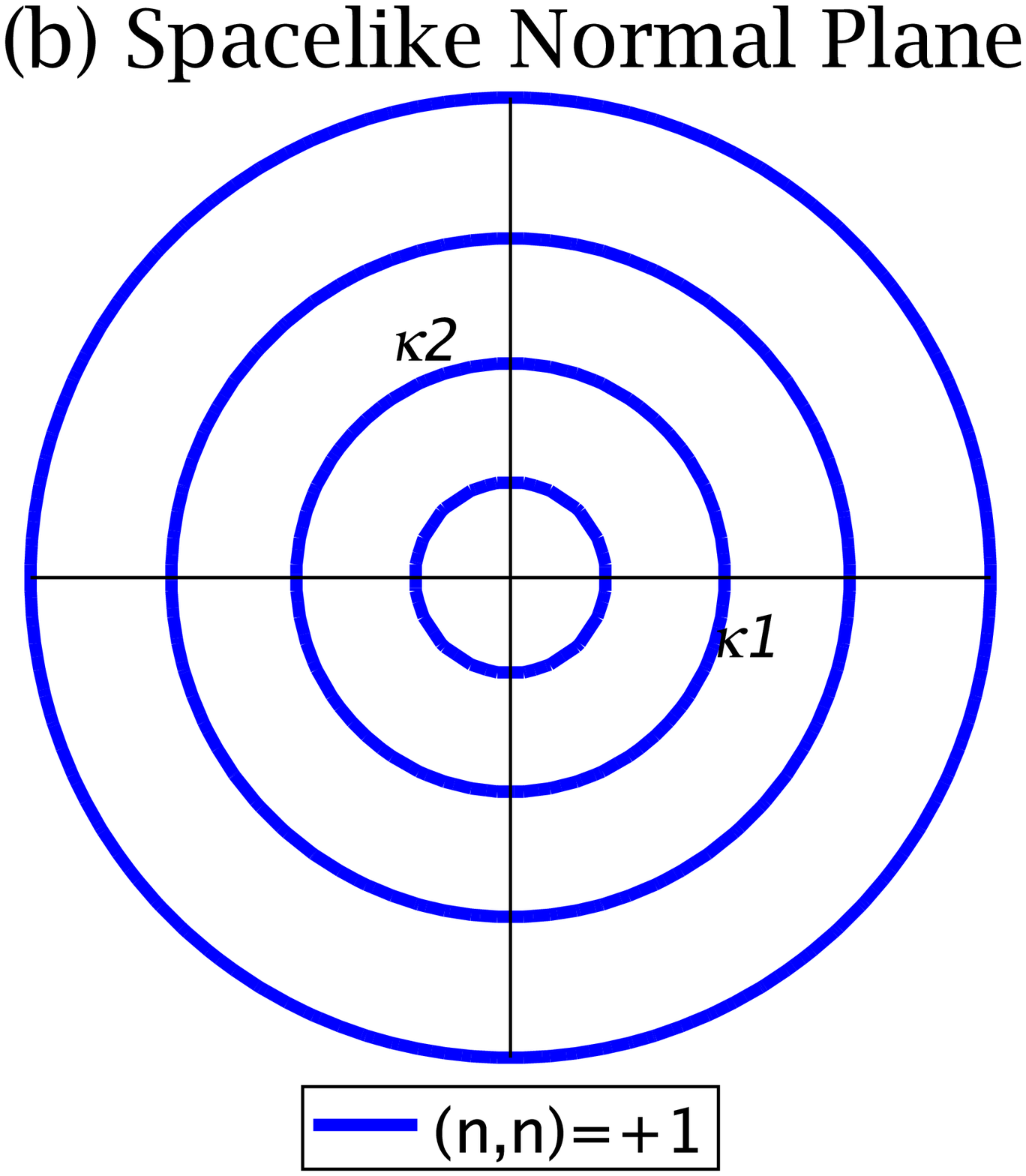}}
    {\includegraphics[width=0.32\linewidth]{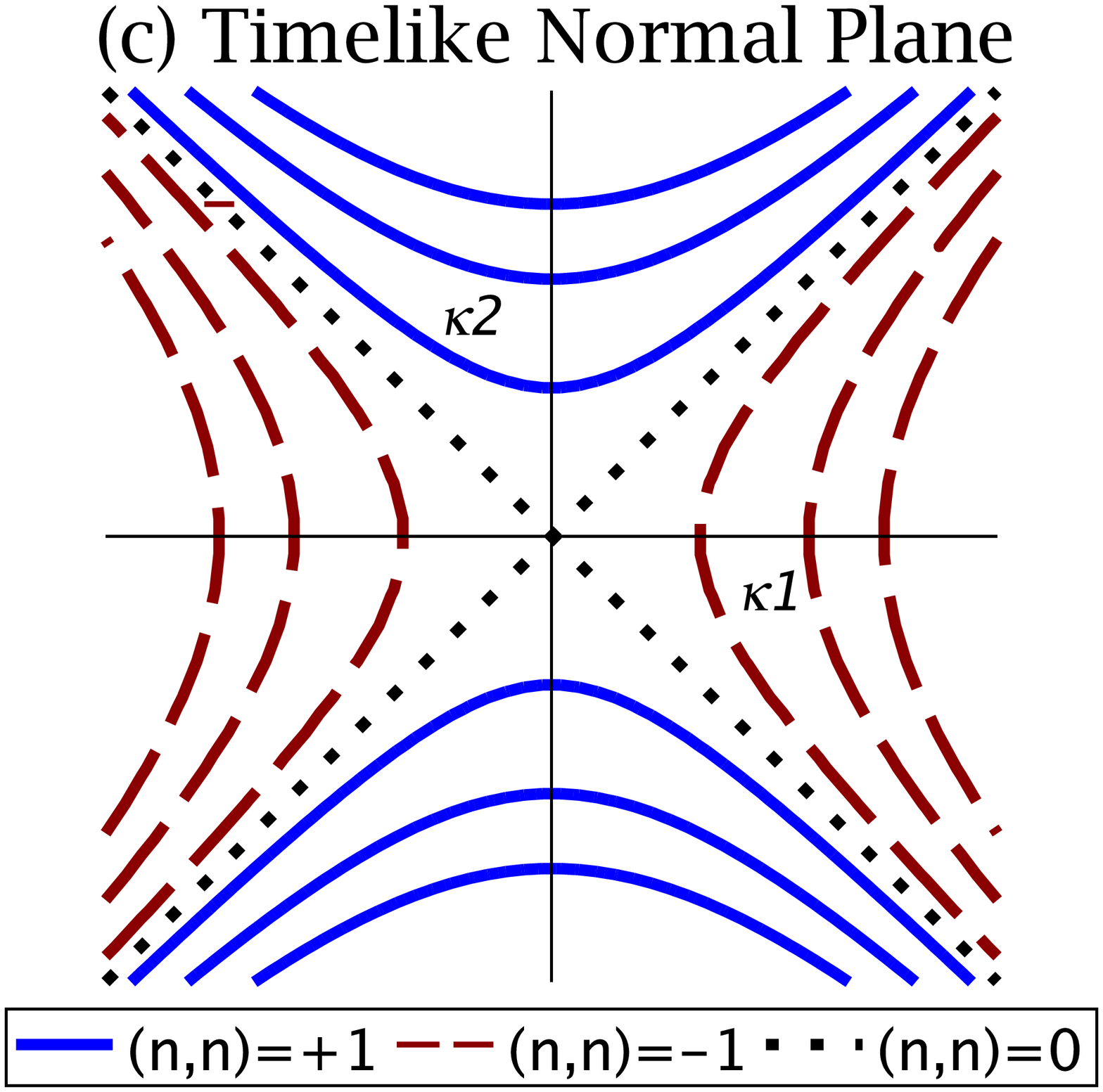}}
          \caption{The geometry of the normal development $(\kappa_1,\kappa_2)$: (a) On a space- or timelike normal plane, lines through the origin (dashed red line) represent planar curves (Proposition \ref{prop::CharacPlaneCurves}), and lines not passing through the origin (solid blue line) represent spherical curves (Section 4); (b) On a spacelike normal plane, circles represent $\kappa$-constant curves; and (c) On a timelike normal plane, hyperbolas represent $\kappa$-constant curves with spacelike normal vector (solid blue line) or timelike normal vector (dashed red line), and the degenerate hyperbola $\kappa_1=\pm \kappa_2$ represents curves with a lightlike normal vector (dotted black line).}
\end{figure*}

The normal development of $\alpha(s)$ is the planar curve $(\kappa_1(s),\kappa_2(s))$. After proving the existence for Bishop moving frames on non-lightlike curves the natural question is how to relate the Bishop curvatures $\kappa_1,\kappa_2$ to the geometry of the curve which defines them.

From the Frenet equations we have
\begin{equation}
\eta\mathbf{n}=\frac{\mathbf{t}'}{\kappa}=\frac{\epsilon_1\kappa_1\mathbf{n}_1+\kappa_2\mathbf{n}_2}{\kappa}\Rightarrow \eta = \epsilon_1\frac{\kappa_1^2}{\kappa^2}+\frac{\kappa_2^2}{\kappa^2}\,,
\end{equation}
where $\eta=(\mathbf{n},\mathbf{n})\in\{-1,0,+1\}$. Then we have the following relations (see figure 1):
\begin{proposition}
For a fixed value of the parameter $s$, the point $(\kappa_1(s),\kappa_2(s))$ lies on a conic. More precisely,
\begin{enumerate}[(a)]
\item If $\mathbf{t}(s)$ is timelike (so $\mathbf{n}(s)$ must be spacelike), then $(\kappa_1(s),\kappa_2(s))$ lies on a circle of radius $\kappa(s)$: $\kappa^2=X^2+Y^2$;
\item If $\mathbf{t}(s)$ is spacelike and $\mathbf{n}(s)$ is timelike (spacelike), then $(\kappa_1(s),\kappa_2(s))$ lies on a hyperbola with foci on the $x$ axis ($y$ axis): $\kappa^2=\pm X^2\mp Y^2$;
\item If $\mathbf{t}(s)$ is spacelike and $\mathbf{n}(s)$ is lightlike, then $(\kappa_1(s),\kappa_2(s))$ lies on the line $X=\pm Y$, which form the asymptotes lines of the hyperbolas in item (b).
\end{enumerate}
\label{prop::geomNormalDevelopm}
\end{proposition}

\begin{remark}
Observe that $\kappa$-constant curves are precisely the orbits of $O_1(2)$, the symmetry group of a Lorentzian plane.
\end{remark}

\begin{proposition}
Let $\alpha:I\to E_1^3$ be a $C^2$ regular curve which is not spherical. Then, the curve is planar if and only if its normal development $(\kappa_1(s),\kappa_2(s))$ lies on a straight line passing through the origin.
\label{prop::CharacPlaneCurves}
\end{proposition}
\begin{proof}
Suppose $a\,\kappa_1+b\,\kappa_2=0$ for some constants $a$ and $b$. Defining $\mathbf{x}(s)=a\mathbf{n}_1(s)+b\mathbf{n}_2(s)\in N_{\alpha(s)}=\mbox{span}\{\mathbf{t}(s)\}^{\perp}$. It follows that $\mathbf{x}$ is constant and also that
\begin{equation}
(\alpha,\mathbf{x})' = (\alpha,-(a\epsilon\kappa_1+b\epsilon\kappa_2)\mathbf{t})=-\epsilon(\alpha,\mathbf{t})(a\kappa_1+b\kappa_2)=0\,.
\end{equation}
Thus, $(\alpha(s),\mathbf{x})$ is constant and then $(\alpha(s)-\alpha(s_0),\mathbf{x})=0$. So, $\alpha$ is a planar curve.

Conversely, let $\alpha$ be contained on a plane $(\alpha(s)-\alpha(s_0),\mathbf{x})=0$. Since the tangent $\mathbf{t}$ also belongs to this plane, we can write $\mathbf{x}=a\mathbf{n}_1+b\mathbf{n}_2$ for some constants $a,b$. Then,
\begin{equation}
0=(\alpha-\alpha_0,\mathbf{x})' = (\alpha-\alpha_0,-(a\epsilon\kappa_1+b\epsilon\kappa_2)\mathbf{t})=-\epsilon(a\kappa_1+b\kappa_2)\,(\alpha-\alpha_0,\mathbf{t})\,.
\end{equation}
Thus, $a\kappa_1+b\kappa_2=0$ or $(\alpha-\alpha_0,\mathbf{t})=0$. In this last case, the curve would be spherical. Indeed, if it were $(\alpha-\alpha_0,\mathbf{t})=0$, then $\alpha-\alpha_0=b_1\mathbf{n}_1+b_2\mathbf{n}_2$, for some constants $b_1$ and $b_2$, because $b_i=\epsilon_i(\alpha-\alpha_0,\mathbf{n}_i)$ and $b_i'=\epsilon_i(\mathbf{t},\mathbf{n}_i)-\epsilon\epsilon_i\kappa_i(\alpha-\alpha_0,\mathbf{t})=0$. 
Taking the derivative of $(\alpha-\alpha_0,\mathbf{t})=0$, gives $0=(\mathbf{t},\mathbf{t})+(\alpha-\alpha_0,\epsilon_1\kappa_1\mathbf{n}_1+\kappa_2\mathbf{n}_2)=\epsilon+b_1\kappa_1+b_2\kappa_2$. But $0=1+\epsilon b_1\kappa_1+\epsilon b_2\kappa_2$ is the equation of a spherical curve (see theorems on Section 4 below).  
\qed
\end{proof}

\begin{remark}
If the pseudo-torsion of a spacelike curve with a lightlike normal vanishes, then the curve is planar (the converse is not true. Indeed, L\'opez \cite{LopesIEJG2014} gives an example of a curve which is planar and has a non-zero pseudo-torsion). However, it follows from the above propositions that all spacelike curves with a lightlike normal are planar, no matter the value of the pseudo-torsion.
\end{remark}


\subsection{Moving frames along lightlike curves}

It is not possible to define a Bishop frame along lightlike curves, we can not even define an orthonormal frame. In this case we must work with the concept of a null frame (see Inoguchi and Lee \cite{InoguchiIEJG2008} for a survey on the geometry of lightlike curves and null frames along them). As in the previous case, we will introduce along $\alpha$ a (null) frame by exploiting the structure of the normal plane only.

Let $\alpha:I\to E_1^3$ be a lightlike curve. In this case, since $\alpha'$ is a lightlike vector, the normal plane $N_{\alpha(s)}=\mbox{span}\{\alpha'(s)\}^{\perp}$ is lightlike and $\alpha'\in N_{\alpha}$. So, we have $N_{\alpha(s)}=\mbox{span}\{\alpha'(s),\mathbf{z}_1(s)\}$, where $\mathbf{z}_1$ is a unit spacelike vector. Denote by $\mathbf{t}=\alpha'$ the tangent vector. If $\mathbf{t}'$ is spacelike, then we can assume $\alpha$ parametrized by pseudo arc-length. Let $\mathbf{z}_2$ be the lightlike vector orthogonal to $\mathbf{z}_1$ and satisfying $(\mathbf{t},\mathbf{z}_2)=-1$. In this case, the equations of motion are
\begin{equation}
\frac{{\rm d}}{{\rm d}s}\left(
\begin{array}{c}
\mathbf{t}\\
\mathbf{z}_1\\
\mathbf{z}_2\\
\end{array}
\right)=\left(
\begin{array}{ccc}
\kappa_3 & \kappa_1 & 0\\
-\kappa_2 & 0 & \kappa_1\\
0 & -\kappa_{2} & -\kappa_3\\ 
\end{array}
\right)\left(
\begin{array}{c}
\mathbf{t}\\
\mathbf{z}_1\\
\mathbf{z}_2\\
\end{array}
\right)=\left(
\begin{array}{ccc}
0 & \kappa_1 & -\kappa_3\\
-\kappa_1 & 0 & \kappa_2\\
\kappa_3 & -\kappa_{2} & 0\\ 
\end{array}
\right)E_{\mathbf{t},\mathbf{z}_1,\mathbf{z}_2}\left(
\begin{array}{c}
\mathbf{t}\\
\mathbf{z}_1\\
\mathbf{z}_2\\
\end{array}
\right),\label{eq::FrameEqsLightCurves}
\end{equation}
where $\kappa_1=(\mathbf{t}',\mathbf{z}_1)$, $\kappa_2=(\mathbf{z}_1',\mathbf{z}_2)$, and $\kappa_3=(\mathbf{z}_2',\mathbf{t})$. Here $E_{\mathbf{t},\mathbf{n},\mathbf{b}}=[(\mathbf{e}_i,\mathbf{e}_j)]_{ij}$ denotes the matrix associated with the null frame $\{\mathbf{e}_0=\mathbf{t},\mathbf{e}_1=\mathbf{n},\mathbf{e}_2=\mathbf{b}\}$. The coefficient $\kappa_1$ plays a significant role on the theory of moving frames along lightlike curves.
\begin{remark}
If $\mathbf{t}'$ is spacelike and if we take $\mathbf{z}_1=\mathbf{t}'=\mathbf{n}$, then $\mathbf{z}_2=\mathbf{b}$ and $\kappa_1=1$, $\kappa_2=\tau$ and $\kappa_3=0$. However, the Frenet frame is not defined when $\mathbf{t}'$ is lightlike. Here, the presence of $\kappa_3$ allows for a description of lightlike curves regardless of the casual character of $\mathbf{t}'$. 
\end{remark}

\begin{proposition}
A lightlike curve $\alpha:I\to E_1^3$ is a straight line if and only if $\kappa_1=0$. Moreover, if $\alpha$ is not a straight line and is parametrized by the pseudo arc-length, then $\kappa_1^2=1$.
\label{prop::CharacLightlikeLine}
\end{proposition}
\begin{proof}
If $\kappa_1=0$, then $\mathbf{t}'=\kappa_3\mathbf{t}$. Integration of this equation gives $\alpha=\alpha_0+(\int {\rm e}^{\int\kappa_3})\mathbf{t}_0$, where $\alpha_0$ and $\mathbf{t}_0$ are constants. Then, $\alpha$ is a straight line. Conversely, let $\alpha=\alpha_0+f\,\mathbf{t}_0$, with $f$ a smooth function. Taking derivatives, it is easy to verify that $\kappa_1=0$.

Now, suppose $\kappa_1\not=0$, so if $\alpha$ is parametrized by pseudo arc-length we have
\begin{equation}
1=(\mathbf{t}',\mathbf{t}')=\kappa_1^2,
\end{equation}
as expected.
\qed
\end{proof}

\section{Characterization of spherical curves in $E_1^3$}

In $E^3$ the function $F(x)=\langle x-P,x-P\rangle$ is non-negative. A sphere of radius $r$ and center $P$ in $E^3$, $\mathbb{S}^2(P;r)$, is then defined as the level sets of $F$, i.e. $\langle x-P,x-P\rangle=r^2$ (if $r=0$ the sphere degenerates to a single point). On the other hand, in $E_1^3$ the function $F_1(x)=(x-P,x-P)$ may assume any value on the real numbers. So, in $E_1^3$ we still define spheres as the level sets of $F_1$, but one must consider three types of spheres, depending on the sign of $F_1$. We shall adopt the following standard notations:
\begin{equation}
\mathbb{S}_1^2(P;r) =\{x\in E_1^3\,:\, (x-P,x-P)=r^2\},\\
\end{equation}
\begin{equation}
\mathcal{C}^2(P) = \{x\in E_1^3\,:\,(x-P,x-P)=0\},
\end{equation}
and
\begin{equation}
\mathbb{H}_0^2(P;r) =\{x\in E_1^3\,:\, (x-P,x-P)=-r^2\},
\end{equation}
where $r\in (0,\infty)$. These spheres are known as pseudo-sphere, light-cone, and pseudo-hyperbolic space, respectively. As surfaces in $E_1^3$ pseudo-spheres and pseudo-hyperbolic spaces have constant Gaussian curvature $1/r^2$ and $-1/r^2$ \cite{LopesIEJG2014}, respectively\footnote{If we see them as surfaces in $E^3$, their Gaussian curvatures are not constant and, additionally, for $\mathbb{S}_1^2(P;r)$ it is negative, while for $\mathbb{H}_0^2(P;r)$ it is positive.}. 

It is well known that the Minkowski metric restricted to $\mathbb{H}_0^2(P;r)$ is a positive definite metric. Then, it follows that $\mathbb{H}_0^2(P;r)$ is a spacelike surface and, consequently, there is no lightlike or timelike curves in $\mathbb{H}_0^2(P;r)$. On the other hand, light-cones are lightlike surfaces \cite{LopesIEJG2014} and, consequently, there is no lightlike curves on them. The pseudo-sphere is the only one that has the three types of curves \cite{InoguchiIEJG2008,LopesIEJG2014}:
\begin{lemma}
There exist no time- and lightlike curves in $\mathbb{H}_0^2(P;r)$ and no timelike curves in $\mathcal{C}^2(P)$.
\end{lemma}

Now we generalize Bishop's characterization of spherical curves in $E^3$ \cite{BishopMonthly} to the context of spheres in $E_1^3$.

\begin{theorem}
A $C^2$ regular spacelike or timelike curve $\alpha:I\to E_1^3$ lies on a sphere of nonzero radius, i.e., $\alpha\subseteq \mathbb{H}_0^2(P;r)$ or $\mathbb{S}_1^2(P;r)$, if and only if its normal development, i.e., the curve $(\kappa_1(s),\kappa_2(s))$, lies on a line not passing through the origin. Moreover, the distance of this line from the origin, $d$, and the radius of the sphere are reciprocals: $d=1/r$.
\label{theo::characSpaceAndLightCurves}
\end{theorem}
\begin{remark}
When a curve is spacelike the normal plane is timelike and then the distance in the normal development plane should be understood as the distance induced by the restriction of $(\cdot,\cdot)$ on the normal plane. So, circles in this plane metric are hyperbolas.
\end{remark}
\begin{proof}
{\it of Theorem} \ref{theo::characSpaceAndLightCurves}. Denote by $\mathcal{Q}$ a sphere $\mathbb{H}_0^2(P;r)$ or $\mathbb{S}_1^2(P;r)$. If $\alpha$ lies in $\mathcal{Q}$, then taking the derivative of $(\alpha-P,\alpha-P)=\pm r^2$ gives 
\begin{equation}
(\alpha-P,\mathbf{t})=0.\label{eq::aux1} 
\end{equation}
This implies that $\alpha-P=a_1\mathbf{n}_1+a_2\mathbf{n}_2$. Now, let us investigate the coefficients $a_i$. Since $a_i=\epsilon_i(\alpha-P,\mathbf{n}_i)$, where $\epsilon_i=(\mathbf{n}_i,\mathbf{n}_i)$, we have
\begin{eqnarray}
a_i' & = & \epsilon_i(\mathbf{t},\mathbf{n}_i)+\epsilon_i(\alpha-P,\mathbf{n}_i')=0\,.
\end{eqnarray}
Therefore, the coefficients $a_1$ and $a_2$ are constants. Finally, taking the derivative of Eq. (\ref{eq::aux1}), we find
\begin{equation}
0=(\mathbf{t},\mathbf{t})+(\alpha-P,\epsilon_1\kappa_1\mathbf{n}_1+\kappa_2\mathbf{n}_2)=\epsilon+a_1\kappa_1+a_2\kappa_2.
\end{equation}
Thus, the normal  development $(\kappa_1,\kappa_2)$ lies on a straight line $1+\epsilon a_1X+\epsilon a_2Y=0$ not passing through the origin. If $Q=\mathbb{S}_1^2(P;r)$, then $r^2=(\alpha-P,\alpha-P)=\epsilon_1 a_1^2+a_2^2=1/d^2$, where $d$ is the distance of the line from the origin. On the other hand, if $Q=\mathbb{H}_0^2(P;r)$, then the curve is necessarily spacelike and $\epsilon_1=-1$, since $\mathbf{n}_1$ is timelike (as mentioned before, $\mathbb{H}_0^2(P;r)$ is a spacelike surface). So, we have $r^2=-(\alpha-P,\alpha-P)= a_1^2-a_2^2=\pm1/d^2$ (the orientation of the hyperbolas will depend on the casual character of the normal vector $\mathbf{n}$ according to Proposition \ref{prop::geomNormalDevelopm}: see figure 1).

Conversely, assume that $0=1+\epsilon a_1\kappa_1+\epsilon a_2\kappa_2$ for some constants $a_1$ and $a_2$.  Define the point $P=\alpha-a_1\mathbf{n}_1-a_2\mathbf{n}_2$. Then $P'=\mathbf{t}+(a_1\epsilon\kappa_1+a_2\epsilon\kappa_2)\mathbf{t}=0$ and therefore $P$ is a fixed point. It follows that $\alpha$ lies on a sphere of nonzero radius and center $P$: $(\alpha-P,\alpha-P)=\epsilon_1a_1^2+a_2^2$.
\qed
\end{proof}

For spacelike curves on light-cones (as mentioned before there is no timelike curve on light-cones: Lemma 1) we have an analogous characterization:
\begin{theorem}
A $C^2$ regular spacelike curve $\alpha:I\to E_1^3$ lies on a light-cone $\mathcal{C}^2(P)$, i.e., lies on a sphere of zero radius, if and only if its normal development, i.e., the curve $(\kappa_1(s),\kappa_2(s))$, lies on a line $\{a_1X+a_2Y+1=0\}$ not passing through the origin. Moreover, we have the relation $a_2=\pm a_1$.
\end{theorem}
\begin{proof}
Let $\alpha$ be a curve in $\mathcal{C}^2(P)$ with $(\mathbf{t},\mathbf{t})=1$ and $(\mathbf{n}_1,\mathbf{n}_1)=-1$, i.e., $\epsilon=1$ and $\epsilon_1=-1$. Now taking the derivative of $(\alpha-P,\alpha-P)=0$ gives 
\begin{equation}
(\alpha-P,\mathbf{t})=0.\label{eq::aux2} 
\end{equation}
This implies that $\alpha-P=a_1\mathbf{n}_1+a_2\mathbf{n}_2$. Since $a_i=\epsilon_i(\alpha-P,\mathbf{n}_i)$, where $\epsilon_i=(\mathbf{n}_i,\mathbf{n}_i)$, we have
\begin{eqnarray}
a_i' & = & \epsilon_i(\mathbf{t},\mathbf{n}_i)+(\alpha-P,\mathbf{n}_i')=0\,.
\end{eqnarray}
Therefore, the coefficients $a_1$ and $a_2$ are constants. Finally, taking the derivative of Eq. (\ref{eq::aux2}), we find
\begin{equation}
0=(\mathbf{t},\mathbf{t})+(\alpha-P,-\kappa_1\mathbf{n}_1+\kappa_2\mathbf{n}_2)=1+a_1\kappa_1+a_2\kappa_2.
\end{equation}
Thus, the normal  development $(\kappa_1(s),\kappa_2(s))$ lies on a straight line $1+a_1X+a_2Y=0$ not passing through the origin. Moreover, $0=(\alpha-P,\alpha-P)=- a_1^2+a_2^2$, which implies $a_2=\pm a_1$.

Conversely, assume that $0=1+ a_1\kappa_1\pm a_1\kappa_2$ for some constant $a_1$. Define the point $P=\alpha-a_1\mathbf{n}_1\mp a_1\mathbf{n}_2$, which satisfies $P'=\mathbf{t}+(a_1\kappa_1\pm a_1\kappa_2)\mathbf{t}=0$. In other words, $P$ is a fixed point and it follows that $\alpha$ lies on a light-cone $\mathcal{C}^2(P)$ of center $P$.
\qed
\end{proof}

For lightlike curves we are not able to use a Bishop frame.  However, by using null frames, we can still state a criterion for a lightlike curve be contained on pseudo-spheres or light-cones (trying to follow steps as in the previous cases does not work, due to the lack of good orthogonality properties). In fact, the following results are generalizations of those of Inoguchi and Lee \cite{InoguchiIEJG2008} for pseudo-spherical lightlike curves.
\begin{theorem}
If a $C^2$ regular lightlike curve $\alpha:I\to E_1^3$ lies on a pseudo-sphere or a light-cone, then $\kappa_1=0$ or, equivalently, $\alpha$ is a straight line.
\end{theorem}
\begin{proof}
Let $\mathcal{Q}$ be a sphere of non-negative radius denoted by $\mathcal{Q}=\{x\,:\,(x-P,x-P)=\rho\}$ where $\rho=r^2$ ($r>0$) or $0$, i.e., $\mathcal{Q}$ is a pseudo-sphere $\mathbb{S}_1^2(P;r)$ or a light-cone $\mathcal{C}^2(P)$. If $\alpha\subseteq\mathcal{Q}$, taking the derivative of $(x-P,x-P)=\rho$ gives
\begin{equation}
(\mathbf{t},x-P)=0.\label{eq::auxtx-pzero}
\end{equation}
Deriving the above equation gives
\begin{equation}
\kappa_1(\mathbf{z}_1,x-P)=0.
\end{equation}
If $\kappa_1$ were not zero, then we would find $(\mathbf{z}_1,x-P)=0$, which by taking a derivative again gives $(\mathbf{z}_2,x-P)=0$. From these two last equations, and from Eq. (\ref{eq::auxtx-pzero}), we would conclude that $x-P=0$, which is not possible. In short, the curve must satisfy $\kappa_1=0$. Finally, by Proposition \ref{prop::CharacLightlikeLine} it follows that $\alpha$ must be a straight line.
\qed
\end{proof}
\begin{remark}
Surfaces in a semi-Riemannian manifold $M_1^3$ have an interesting property: a lightlike curve is always a pregeodesic, i.e., there exists a parametrization that makes the curve a parametrized geodesic \cite{ONeill}. In $\mathbb{R}^3$ equipped with the standard Minkowski metric, a lightlike curve is a geodesic if and only if it is straight line \cite{InoguchiIEJG2008}.
\end{remark}

The converse of the above theorem is not true. In fact, taking $(\cdot,\cdot)$ as the standard Minkowski metric, the straight line $\alpha(\tau)=(0,0,\tau)$ does not lie on any pseudo-sphere or light-cone. However, we have the following partial converse: 
\begin{proposition}
Let $\alpha_0\in \mathcal{Q}(P;\rho)=\{x:(x-P,x-P)=\rho\}$ be a point on a pseudo-sphere or light-cone, i.e., $\rho=r^2$ ($r>0$) or $=0$. If $\mathbf{u}\in T_{\alpha_0}\mathcal{Q}(P;\rho)$ is a lightlike vector, then for any smooth function $f(\tau)$ the curve $\alpha(\tau)=\alpha_0+f(\tau)\,\mathbf{u}$ is a lightlike straight line that lies on $\mathcal{Q}(P;\rho)$.
\end{proposition}
\begin{proof}
Using that $\mathbf{u}\in T_{\alpha_0}\mathcal{Q}(P;\rho)$ implies $(\alpha_0-P,\mathbf{u})=0$, we find
\begin{eqnarray}
(\alpha-P,\alpha-P) & = & (\,(\alpha_0-P)+f\,\mathbf{u},(\alpha_0-P)+f\,\mathbf{u})\nonumber\\
& = & (\alpha_0-P,\alpha_0-P)=\rho\,.
\end{eqnarray}
So, the desired result follows.
\qed
\end{proof}

\section{Characterization of curves on Euclidean quadrics}

Quadrics are the simplest examples of level surfaces and understanding how the characterization works in this particular instance will prove very useful. Indeed, it will become clear in the following that the proper geometric setting to attack the characterization problem on a surface $\Sigma=F^{-1}(c)$ is that of a metric induced by the Hessian of $F$.

Points on a quadratic surface $\mathcal{Q}\subset\mathbb{R}^3$ can be characterized by a symmetric matrix $B\in\mbox{M}_{3\times3}(\mathbb{R})$ as
\begin{equation}
x\in\mathcal{Q}\Leftrightarrow\langle \,B(x-P),x-P\,\rangle=r^2,\label{eq:QuadraticSurface}
\end{equation}
where $P$ is a fixed point (the center of $\mathcal{Q}$), $r>0$ is a constant (the radius of $\mathcal{Q}$), and $\langle\cdot,\cdot\rangle$ is the canonical inner product on $\mathbb{R}^3$. Naturally, if the symmetric matrix $B$ has a non-zero determinant, then this non-degenerate quadric induces a metric or a pseudo-metric on $\mathbb{R}^3$ by defining
\begin{equation}
(\cdot,\cdot)=\langle B\,\cdot,\cdot\rangle\,.
\end{equation}

If the matrix $B$ has index 0, then $\mathcal{Q}$ is an ellipsoid and it can be seen as a sphere on the 3-dimensional Riemannian manifold $M^3=(\mathbb{R}^3,\langle B\,\cdot,\cdot\rangle)$. The characterization of those spatial curves that belong to an ellipsoid can be made through a direct adaption of Bishop's characterization of spherical curves in $E^3$ \cite{Etayo2016}. Indeed, one just uses the metric $\langle B\,\cdot,\cdot\rangle$ instead of $\langle\cdot,\cdot\rangle$ and then follows the steps on the construction of a Bishop frame in $E^3$. On the other hand, if the matrix $B$ has index 1, then $\mathcal{Q}$ is a one-sheeted hyperboloid and can be seen as a pseudo-sphere on a Lorentz-Minkowski space $E^3_1=(\mathbb{R}^3,\langle B\,\cdot,\cdot\rangle)$. If $B$ has index 2, $\mathcal{Q}$ is then a two-sheeted hyperboloid and can be seen as a pseudo-hyperbolic plane on a Lorentz-Minkowski space $E^3_1=(\mathbb{R}^3,\langle -B\,\cdot,\cdot\rangle)$. This way, the results on the previous section can be applied in order to characterize those spatial curves that belong to a (one or two-sheeted) hyperboloid.

Since the characterization of curves on a quadric is made be reinterpreting the problem on a new geometric setting, a natural question then arises: {\it How do we interpret the casual character that a spatial curve assumes when we pass from $E^3$ to $E_1^3$?}

This question can be answered if we take into account the following expression for the normal curvature on a level surface $\Sigma=F^{-1}(c)$ \cite{DombrowskiMN1968}
\begin{equation}
\kappa_n(p,\mathbf{v}) = \frac{\langle \mbox{Hess}_pF\,\mathbf{v},\mathbf{v}\rangle}{\Vert\nabla_p F\Vert},\label{eqNormalCurvLevelSets}
\end{equation}
where $\mathbf{v}\in T_p\Sigma$, and $\mbox{Hess}\,F$ and  $\nabla F$ are the Hessian and the gradient vector of $F$, respectively (for more details involving the expressions for the curvatures of level set surfaces see \cite{GoldmanCAGD2005}). Then, we have the following interpretation:
\begin{proposition}
If $\alpha:I\to\mathbb{R}^3$ is a curve on a non-degenerate quadric $\mathcal{Q}$, then asymptotic directions (in $\mathcal{Q}\subseteq E^3$) correspond to lightlike directions (in $\mathcal{Q}\subseteq E_1^3$).
\label{prop::InterpretCasualChar}
\end{proposition}
\begin{proof}
Quadrics are level sets of $F(x)=\langle B\,(x-P),x-P\rangle$ and $\mbox{Hess}\,F=B$. Now, since the quadric is non-degenerate, we have that $\mathcal{Q}$ is the inverse image of a regular value of $F$. So, we can apply Eq. (\ref{eqNormalCurvLevelSets}).
\qed
\end{proof}

Based on these constructions we can better interpret why pseudo-spheres $\mathbb{S}_1^2$ have both space- and timelike tangent vectors, while pseudo-hyperbolic planes $\mathbb{H}_0^2$ only have spacelike ones. Indeed, Eq. (\ref{eqNormalCurvLevelSets}) shows that the sign of the Gaussian curvature (in $E^3$), $K_{E^3}$, has an impact on the casual character of the tangent plane:  points with $K_{E^3}>0$ have spacelike tangent planes, while points with $K_{E^3}<0$ have timelike tangent planes. 

Finally, observe that quadrics are level sets of  $F(x)=\langle B(x-P),x-P\rangle$, which has a constant Hessian: $\mbox{Hess}\,F=B$. This motivates us to consider this procedure for any level surface.

\section{Curves on level surfaces of a smooth function}

Let $\Sigma$ be a surface implicitly defined by a smooth function $F:U\subseteq\mathbb{R}^3\to\mathbb{R}$. Then, the Hessian of $F$ induces on $\mathbb{R}^3$ a (pseudo-) metric 
\begin{equation}
(\cdot,\cdot)_p = \langle\mbox{Hess}_p\,F\,\cdot\,,\cdot\rangle=\left\langle\frac{\partial^2F(p)}{\partial x^i\partial x^j}\,\cdot\,,\cdot\right\rangle\,.\label{eq::HessMetric}
\end{equation}
By using Eq. (\ref{eqNormalCurvLevelSets}), Proposition \ref{prop::InterpretCasualChar} is still valid for $\Sigma$ in the context of a Hessian pseudo-metric. Moreover, if $\det(\mbox{Hess}_pF)\not=0$, then $\mbox{Hess}\,F$ is non-degenerate on a neighborhood of $p$. Likewise, since the eigenvalues vary continuously \cite{SerreMatrixBook} and the index can be seen as the number of negative eigenvalues, the Hessian $\mbox{Hess}\,F$ has a constant index on an open neighborhood. Then, $(\cdot,\cdot)$ in Eq. (\ref{eq::HessMetric}) is well defined on a neighborhood of a non-degenerate point $p$. 

Now we ask ourselves if the techniques developed in the previous sections can be applied to characterize curves that lie on a level surface. Unhappily, we are not able to establish a characterization via a linear equation as previously done. Nonetheless, we can still exhibit a functional relationship between the curvatures $\kappa_1$ and $\kappa_2$ of a Bishop frame of the corresponding curves with respect to the Hessian metric. Before that, let us try to understand the technical difficulties involved in the study of level surfaces:

\begin{example}[index 1 Hessian]
Suppose that $\mbox{index}(\mbox{Hess}\,F)=1$ on a certain neighborhood of a non-degenerate point $p$. Let $\alpha:I\to E^3$ be a curve on a regular level surface $\Sigma = F^{-1}(c)$ whose velocity vector $\alpha'\in T_{\alpha(s)}\Sigma$ is not an asymptotic direction for all $s\in I$, i.e. $\kappa_n(\alpha(s),\mathbf{\alpha}'(s))\not=0$. This means that the curve is timelike or spacelike. Denote by $\{\mathbf{t},\mathbf{n}_1,\mathbf{n}_2\}$ a Bishop frame along $\alpha$, with respect to Eq. (\ref{eq::HessMetric}), and denote by $D$ the covariant derivative and by a prime $'$ the usual one. 

From $F(\alpha(s))=c$ it follows that
\begin{equation}
(\mbox{grad}_{\alpha(s)}F,\mathbf{t}) = 0\Rightarrow\mbox{grad}_{\alpha}F=a_1\mathbf{n}_1+a_2\mathbf{n}_2\,,\label{eq:gradF_hessMet_InNormalPlane}
\end{equation}
where $\mbox{grad}_{\alpha}F$ denotes the gradient vector with respect to $(\cdot,\cdot)$. The coefficients $a_1$ and $a_2$ satisfy $a_i=\epsilon_i(\mbox{grad}_{\alpha}F,\mathbf{n}_i)$ and, therefore,
\begin{eqnarray}
\epsilon_ia_i' & = & (D\,\mbox{grad}_{\alpha}F, \mathbf{n}_i)+(\mbox{grad}_{\alpha}F,D\,\mathbf{n}_i)\nonumber\\
& = & H^F(\mathbf{t}, \mathbf{n}_i)-\epsilon\kappa_i(\mbox{grad}_{\alpha}F,\mathbf{t})\nonumber\\
& = & H^F(\mathbf{t}, \mathbf{n}_i),
\end{eqnarray}
where $H^F$ denotes the Hessian with respect to $(\cdot,\cdot)$, whose coefficients can be expressed as \cite{ONeill}
\begin{equation}
H^F_{ij} = \left(\frac{\partial^2F}{\partial x^i\partial x^j}-\sum_k\Gamma_{ij}^k\frac{\partial F}{\partial x^k}\right)\,.
\end{equation}
From this expression we see that $a_i'$ does not need to be zero and then we can not apply the same steps as in the previous sections. Indeed, the orthogonality of the Bishop frame $\{\mathbf{t},\mathbf{n}_1,\mathbf{n}_2\}$ with respect to $\mbox{Hess}\,F=\partial^2F/\partial x^i\partial x^j$ and $H^F$ does not coincide, unless $\mbox{Hess}\,F$ is constant. 
\qed
\end{example}

\begin{theorem}
Let $\mathcal{U}_p\subseteq\mathbb{R}^3$ be a neighborhood of a non-degenerate point $p\in\Sigma=F^{-1}(c)$ where the index is constant. Let $H^F$ denotes the Hessian with respect to the Hessian metric $(\cdot,\cdot)_q=\langle{\rm Hess}_qF\,\cdot,\cdot\rangle$.

If $\alpha:I\to\mathcal{U}_p\cap\Sigma$ is a $C^2$ regular curve, with no asymptotic direction for ${\rm index}({\rm Hess}\,F
)\not\in\{0,3\}$, i.e., $\kappa_n(\alpha,\alpha')\not=0$, then its normal development $(\kappa_1(s),\kappa_2(s))$ satisfies
\begin{equation}
a_2(s)\kappa_2(s)+a_1(s)\kappa_1(s)+a_0(s)=0, \label{eq::characLevelSurfacesCurves}
\end{equation}
where $a_0=H^F(\mathbf{t},\mathbf{t})$, $a_i=({\rm grad}_{\alpha}F,\mathbf{n}_i)$, and $a_i'(s)=H^F(\mathbf{t},\mathbf{n}_i)$: or $\epsilon_i H^F(\mathbf{t},\mathbf{n}_i)$, $\epsilon_i=(\mathbf{n}_i,\mathbf{n}_i)=\pm1$, if ${\rm index}({\rm Hess}\,F
)\not\in\{0,3\}$. Here, the Bishop frame is defined with respect to the Hessian metric.

Conversely, if Eq. (\ref{eq::characLevelSurfacesCurves}) is valid and $({\rm grad}_{\alpha(s_0)}F,\mathbf{t}(s_0))=0$ at some point $\alpha(s_0)$, then $\alpha$ lies in a level surface of $F$.
\label{theo::CurvesInLevelSets}
\end{theorem}
\begin{remark}
If $\Sigma=F^{-1}(c)$, where $c$ is a regular value of $F$, then $\Sigma$ is an orientable surface. The reciprocal of this result is also valid, i.e., every orientable surface is the inverse image of a regular value of some smooth function \cite{Guillemin}. Then, the above theorem can be applied to any orientable surface (we still have to exclude those points where the Hessian has a zero determinant).
\end{remark}
\begin{proof}
{\it of theorem }\ref{theo::CurvesInLevelSets}.
If the index is $0$, then the Hessian metric defines a Riemannian metric (if $\mbox{index}(\mbox{Hess}\,F)=-3$, then its negative defines a metric). On the other hand, the construction of a Bishop frame for a pseudo-metric with index 2 in dimension 3 is completely analogous to the case of index 1. Moreover, when the index of $\mbox{Hess}\,F$ is 1 (or 2), the assumption that $\alpha'$ is not an asymptotic direction means that $\alpha$ must be a space- or a timelike curve.  

In the following, let us assume that ${\rm index}(\mbox{Hess}\,F)=1$, the other cases being analogous. In this case, Eq. (\ref{eq::HessMetric}) defines a pseudo-metric in $\mathcal{U}_p\subseteq\mathbb{R}^3$.

Since $F(\alpha(s))=c$, we have
\begin{equation}
(\mbox{grad}_{\alpha(s)}F,\mathbf{t}) = 0\Rightarrow\mbox{grad}_{\alpha}F=a_1\mathbf{n}_1+a_2\mathbf{n}_2\,,\label{eqGradFHessMetric}
\end{equation}
where $\mbox{grad}_{\alpha}F$ denotes the gradient vector with respect to $(\cdot,\cdot)$. The coefficients $a_1$ and $a_2$ satisfy $a_i=\epsilon_i(\mbox{grad}_{\alpha}F,\mathbf{n}_i)$ and, therefore,
\begin{eqnarray}
a_i' & = & \epsilon_i(D\,\mbox{grad}_{\alpha}F, \mathbf{n}_i)+\epsilon_i(\mbox{grad}_{\alpha}F,D\,\mathbf{n}_i)\nonumber\\
& = & \epsilon_iH^F(\mathbf{t}, \mathbf{n}_i)-\epsilon_i\epsilon\kappa_i(\mbox{grad}_{\alpha}F,\mathbf{t})\nonumber\\
& = & \epsilon_iH^F(\mathbf{t}, \mathbf{n}_i),
\end{eqnarray}
where $H^F$ denotes the Hessian with respect to $(\cdot,\cdot)$ \cite{ONeill}. Taking the derivative of Eq. (\ref{eqGradFHessMetric}) gives
\begin{eqnarray}
0 & = & (D\,\mbox{grad}_{\alpha}F,\mathbf{t})+(\mbox{grad}_{\alpha}F,D\,\mathbf{t})\nonumber\\
& = & H^F(\mathbf{t},\mathbf{t})+(a_1\mathbf{n}_1+a_2\mathbf{n}_2,\epsilon_1\kappa_1\mathbf{n}_1+\kappa_2\mathbf{n}_2)\nonumber\\
& = & H^F(\mathbf{t},\mathbf{t})+a_1\kappa_1+a_2\kappa_2\,.
\end{eqnarray}
Then, Eq. (\ref{eq::characLevelSurfacesCurves}) is satisfied. 

Conversely, suppose that Eq. (\ref{eq::characLevelSurfacesCurves}) is satisfied. Let us define the function $f(s)=F(\alpha(s))$. We must show that $f$ is constant, i.e., $f'(s)=0$. Taking the derivative of $f$ twice gives
\begin{equation}
f' = (\mbox{grad}_{\alpha}F,\mathbf{t}),
\end{equation}
and
\begin{eqnarray}
f'' & = & (D\,\mbox{grad}_{\alpha}F,\mathbf{t})+(\mbox{grad}_{\alpha}F,D\,\mathbf{t})\nonumber\\
& = & H^F(\mathbf{t},\mathbf{t})+\epsilon_1\kappa_1(\mbox{grad}_{\alpha}F,\mathbf{n}_1)+\kappa_2(\mbox{grad}_{\alpha}F,\mathbf{n}_2)\nonumber\\
& = & 0.
\end{eqnarray}
Then, $f'(s)=(\mbox{grad}_{\alpha(s)}F(s),\mathbf{t}(s))$ is constant. By assumption, we have $f'(s_0)=0$, then $f(s)=F(\alpha(s))$ is constant on an open neighborhood of $s_0$, i.e., $\alpha$ lies on a level surface of $F$.
\qed
\end{proof}

\begin{remark}
The Christoffel symbols $\Gamma_{ij}^k$ of a Hessian metric $g_{ij}=\partial^2F/\partial x^i\partial x^j$ vanish if and only if $\mbox{Hess}\,F$ is constant; which is valid for a quadratic surface, this case being treated in the previous section. 
\end{remark}

If $\mbox{Hess}\,F$ degenerates, i.e., $\det(\mbox{Hess}_pF)=0$ at some points, then the Hessian matrix does not define a metric. Nonetheless, it is still possible to characterize curves on a level surface by using the standard metric of $\mathbb{R}^3$. In fact, it can be used even if $\mbox{Hess}\,F$ is non-degenerate, but in this case we do not have non-degenerate quadrics as a particular instance. The obtained criterion is completely analogous to the previous one in Theorem \ref{theo::CurvesInLevelSets}. Indeed, we have

\begin{theorem}
If $\alpha:I\to E^3\cap \Sigma$ is a $C^2$ regular curve, where $\Sigma=F^{-1}(c)$, then its normal development $(\kappa_1(s),\kappa_2(s))$ satisfies
\begin{equation}
b_2(s)\kappa_2(s)+b_1(s)\kappa_1(s)+b_0(s)=0, \label{eq::characLevelSurfacesCurves2}
\end{equation}
where $b_0=\langle ({\rm Hess}\,F)\,\mathbf{t},\mathbf{t}\rangle$, $b_i=\langle\nabla_{\alpha}F,\mathbf{n}_i\rangle$, and $b_i'(s)=\langle ({\rm Hess}\,F)\,\mathbf{t},\mathbf{n}_i\rangle$. Here, the Bishop frame is defined with respect to the usual metric in $E^3$. 

Conversely, if Eq. (\ref{eq::characLevelSurfacesCurves2}) is valid and $\langle\nabla_{\alpha(s_0)}F,\mathbf{t}(s_0)\rangle=0$ at some point $\alpha(s_0)$, then $\alpha$ lies in a level surface of $F$.
\end{theorem}

\begin{proof}
Let $\{\mathbf{t},\mathbf{n}_1,\mathbf{n}_2\}$ be a Bishop frame along $\alpha:I\to E^3$. If $F(\alpha(s))=c$, then we have
\begin{equation}
\langle\nabla_{\alpha(s)}F,\mathbf{t}\rangle = 0\Rightarrow\nabla_{\alpha}F=b_1\mathbf{n}_1+b_2\mathbf{n}_2\,,\label{eqGradFUsualMetric}
\end{equation}
where $\nabla_{\alpha}F$ denotes the gradient vector with respect to usual metric in $E^3$. The coefficients $b_1$ and $b_2$ satisfy $b_i=\langle\nabla_{\alpha}F,\mathbf{n}_i\rangle$ and, therefore,
\begin{eqnarray}
b_i' = \langle (\mbox{Hess}\,F)\,\mathbf{t}, \mathbf{n}_i\rangle-\kappa_i\langle\nabla_{\alpha}F,\mathbf{t}\rangle=\langle (\mbox{Hess}\,F)\,\mathbf{t}, \mathbf{n}_i\rangle.
\end{eqnarray}
Taking the derivative of Eq. (\ref{eqGradFUsualMetric}) gives
\begin{eqnarray}
0 & = & \langle (\mbox{Hess}\,F)\,\mathbf{t},\mathbf{t}\rangle+\langle b_1\mathbf{n}_1+b_2\mathbf{n}_2,\kappa_1\mathbf{n}_1+\kappa_2\mathbf{n}_2\rangle\nonumber\\
& = & \langle (\mbox{Hess}\,F)\,\mathbf{t},\mathbf{t}\rangle+b_1\kappa_1+b_2\kappa_2\,.
\end{eqnarray}
So, Eq. (\ref{eq::characLevelSurfacesCurves2}) is valid. 

Conversely, suppose that Eq. (\ref{eq::characLevelSurfacesCurves}) is satisfied. Let us define the function $f(s)=F(\alpha(s))$. Taking the derivative of $f$ twice gives
\begin{equation}
f' = \langle\nabla_{\alpha}F,\mathbf{t}\rangle,
\end{equation}
and
\begin{eqnarray}
f''
& = & \langle (\mbox{Hess}\,F)\,\mathbf{t},\mathbf{t}\rangle+\kappa_1\langle\nabla_{\alpha}F,\mathbf{n}_1\rangle+\kappa_2\langle\nabla_{\alpha}F,\mathbf{n}_2\rangle= 0.
\end{eqnarray}
Then, $f'(s)=\langle\nabla_{\alpha(s)}F(s),\mathbf{t}(s)\rangle$ is constant. By assumption, we have $f'(s_0)=0$, then $f(s)=F(\alpha(s))$ is constant on an open neighborhood of $s_0$, i.e., $\alpha$ lies on a level surface of $F$.
\qed
\end{proof}

\section{Discussion and Conclusions}

In this work we were interested in the characterization of curves that lie on a given surface. The main tool to achieve that was the use of moving frames along curves. In the construction of Frenet frames in Lorentz-Minkowski spaces $E_1^3$, we showed that the coefficient matrix of the frame motion can be obtained from a skew-symmetric matrix (precisely the matrix that would appear in an Euclidean context) through a right-multiplication by the matrix that describe a frame $\{\mathbf{e}_0,\mathbf{e}_1,\mathbf{e}_2\}$ as a basis of $E_1^3$: $[(\mathbf{e}_i,\mathbf{e}_j)]_{ij}$. Later, by adapting Bishop's idea of relatively parallel moving frames, we were able to furnish a complete characterization of spherical curves in $E_1^3$ through a linear equation relating the coefficients which dictate the frame motion. To attain that, we developed a systematic approach to the construction of Bishop frames by exploiting the structure of the normal planes induced by the casual character of the curve, while for lightlike curves we made use of null frames. In both cases, the coefficient matrix of the frame motion can be obtained from a skew-symmetric matrix, the matrix that would appear in an Euclidean context, through a right-multiplication by the matrix that describes the frame as a basis. We then applied these ideas to surfaces that are level sets of a smooth function, $\Sigma=F^{-1}(c)$, by reinterpreting the problem in the context of the metric given by the Hessian of $F$, which is not always positive definite. So, we are naturally led to the study of curves in $E_1^3$. We also interpreted the casual character that a curve may assume when we pass from $E^3$ to $E_1^3$ and finally established a criterion for a curve to lie on a level surface of a smooth function, which reduces to a linear equation when the Hessian is constant and happens for non-degenerate Euclidean quadrics.

An interesting problem which remains open is to consider the possibility of a curve changing its casual character. Since the property of being space- or timelike is open, i.e., if it is valid at a point it must be valid on a neighborhood of that point, the real problem is to understand what happens near lightlike points. Moreover, the techniques applied here can be extended to higher dimensions and also to the setting of a Riemannian or a semi-Riemannian manifold $M_{\nu}^n$. Indeed, we believe that it is possible to systematically build Bishop and null frames along curves in $M_{\nu}^n$ as done in this work and then apply these constructions to study level hypersurfaces of a smooth function $F:M_{\nu}^3\to\mathbb{R}$. Since the relation between normal curvature and the Hessian with respect to a (pseudo)-metric is still valid, we can also interpret what happens in the transition from $M_{\nu}^n$ to the new geometric setting of a Hessian metric, which may be of a Lorentzian nature since the Hessian may fail to be positive definite.

To the best of our knowledge, this was the first time that the characterization problem for curves was considered for a large class of surfaces. This makes this work an important contribution to the geometry of curves and surfaces.

\begin{acknowledgements}
The author would like to thank useful discussions with J. Deibsom da Silva and F. A. N. Santos, and also the financial support by Conselho Nacional de Desenvolvimento Cient\'ifico e Tecnol\'ogico - CNPq (Brazilian agency).
\end{acknowledgements}

\end{document}